\documentclass[11pt, twoside]{article}
\usepackage{latexsym}
\usepackage{amsmath}
\usepackage{amssymb}
\usepackage[all]{xy}
\usepackage{amsfonts}
\usepackage{verbatim}
\usepackage{amsthm}
\usepackage{bm}
\usepackage{mathrsfs}
\usepackage{epsfig}
\usepackage{xy}
\usepackage{array}
\usepackage{stmaryrd}
\usepackage{graphicx,color}
\usepackage{xcolor}
\usepackage{tikz}
\usetikzlibrary{arrows,calc}
\usepackage{etex}
\usepackage{mathdots}
\usepackage{float}
\usepackage{graphics}
\usepackage{pdflscape}
\usepackage{CJK}
\usepackage{anysize,hyperref}
\input xypic
\xyoption{all}
\usepackage{extarrows}
\usepackage[perpage,symbol]{footmisc}
\usepackage{setspace}
\allowdisplaybreaks[4]
\renewcommand{\cal}{\mathcal}
\def\A{\mathcal{A}}

\def\U{\mathcal{U}}
\def\V{\mathcal{V}}
\def\W{\mathcal{W}}
\def\S{\mathcal{S}}
\def\T{\mathcal{T}}
\def\P{\mathcal{P}}

\def\E{\mathbb{E}}
\def\B{\mathcal{B}}
\def\C{\mathcal{C}}
\def\s{\mathfrak{s}}

\def\del{\delta}
\def\dr{\ar@{->}[r]}

\def\X{\mathcal{X}}
\def\Y{\mathcal{Y}}

\def\Hom{\mbox{Hom}}

\def\cone{\mbox{Cone}}
\def\cocone{\mbox{CoCone}}

\def\Ker{\mbox{Ker}\hspace{.01in}}
\def\Im{\mbox{Im}\hspace{.01in}}

\begin{document}
\baselineskip=15pt
\title{\Large{\bf Recollements arising from cotorsion pairs on extriangulated\\[2mm] categories$^\bigstar$\footnotetext{\hspace{-1em}$^\bigstar$Yonggang Hu was supported by the China Scholarship Council (CSC) and Panyue Zhou was supported by the National Natural Science Foundation of China (Grant No. 11901190 and 11671221), and by the Hunan Provincial Natural Science Foundation of China (Grant No. 2018JJ3205), and by the Scientific Research Fund of Hunan Provincial Education Department (Grant No. 19B239).}}}
\medskip
\author{Yonggang Hu and Panyue Zhou}

\date{}

\maketitle
\def\blue{\color{blue}}
\def\red{\color{red}}

\newtheorem{theorem}{Theorem}[section]
\newtheorem{lemma}[theorem]{Lemma}
\newtheorem{corollary}[theorem]{Corollary}
\newtheorem{proposition}[theorem]{Proposition}
\newtheorem{conjecture}{Conjecture}
\theoremstyle{definition}
\newtheorem{definition}[theorem]{Definition}
\newtheorem{question}[theorem]{Question}
\newtheorem{remark}[theorem]{Remark}
\newtheorem{remark*}[]{Remark}
\newtheorem{example}[theorem]{Example}
\newtheorem{example*}[]{Example}
\newtheorem{condition}[theorem]{Condition}
\newtheorem{condition*}[]{Condition}
\newtheorem{construction}[theorem]{Construction}
\newtheorem{construction*}[]{Construction}

\newtheorem{assumption}[theorem]{Assumption}
\newtheorem{assumption*}[]{Assumption}

\baselineskip=17pt
\parindent=0.5cm

\begin{abstract}
\baselineskip=16pt
This paper is devoted to constructing some recollements of additive categories associated to concentric twin cotorsion pairs on an extriangulated category. As an application, this result generalizes the work by Chen-Liu-Yang in a triangulated case. Moreover, it highlights new phenomena when it applied to an exact category. Finally, we give some applications to illustrate our main results. In particular, we obtain the Krause's recollement whose the proofs are both elementary and very general.
\\[0.5cm]
\medskip
\end{abstract}

\pagestyle{myheadings}
\markboth{\rightline {\scriptsize Yonggang Hu and  Panyue Zhou}}
         {\leftline{\scriptsize  Recollements arising from cotorsion pairs on extriangulated categories}}

\section{Introduction}
The recollement of triangulated categories was introduced in a geometric setting by Beilinson, Bernstein and Deligne \cite{BBD}. Nowadays it has become very powerful in understanding relationships among three algebraic, geometric or topological objects.
 A fundamental example of the recollement of abelian categories
is due to MacPherson and Vilonen \cite{MV}, in which it first appeared as an inductive step in the construction of perverse sheaves.
Later, Wang and Lin \cite{WL} defined the notion of recollement of addtive categories, which unifies the recollment of abelian categories and the recollement of
triangulated categories.
Recently, Chen, Liu and Yang \cite{CLY} introduced localization sequences and colocalization
sequences of additive categories which are similar to lower recollements and upper
recollements of triangulated categories respectively. They consider recollements of additive categories by cotorsion pairs in triangulated categories. Based on
this fact that a co-$t$-structure is a special cotorsion pair, they give some recollements
of additive categories associated to concentric twin cotorsion pairs in a triangulated
category by doing quotient.

Extriangulated categories were recently introduced by Nakaoka and Palu \cite{NP} by extracting
those properties of ${\rm Ext}^1$ on exact categories (which is itself a generalisation of the concept of a module category and an abelian category) and on triangulated categories that seem relevant from the point of view of cotorsion pairs. In particular, exact categories and triangulated
categories are extriangulated categories. There are a lot of examples of extriangulated categories which are neither exact  categories nor triangulated categories, see \cite{NP,ZZ}.
Hence, many results hold on exact categories and triangulated categories can be unified in the
same framework. Based on this idea, we extend Chen-Liu-Yang's results to extriangulated categories.

Our main result is the following.
\begin{theorem} \emph{(See Theorem \ref{main1} for more details)}
Let $\C$ be an extriangulated category with enough projectives and enough injectives, and $(\S,\T), (\U,\V)$ and $(\X,\Y)$ be cotorsion pairs on  $\C$ with
$\W=\S\cap\T = \U\cap\V =\X\cap\Y$. If $\X\cap\T = \V, \Y\subseteq\T, \Sigma\V\subseteq\V$ and
 $\Sigma\Y\subseteq\Y$, then there
are two recollements of additive categories as follows:
$$\xymatrixcolsep{5pc}\xymatrix{
&\V/\W\ar[r]^{F}
&\T/\W\ar@/_1.5pc/[l]_{F_{\lambda}} \ar@/^1.5pc/[l]^{F_{\del}} \ar[r]^{G}
&(\T\cap\U)/\W\ar@/^1.5pc/[l]^{G_{\del}}\ar@/_1.5pc/[l]_{G_{\lambda}}},$$
$$\xymatrixcolsep{5pc}\xymatrix{
&\V/\W\ar[r]^{F}
&\T/\W\ar@/_1.5pc/[l]_{F_{\lambda}} \ar@/^1.5pc/[l]^{F_{\del}} \ar[r]^{G'}
&\Y/\W\ar@/^1.5pc/[l]^{G'_{\del}}\ar@/_1.5pc/[l]_{G'_{\lambda}}},$$
where $F, G_{\lambda}$ and $G_{\del}$ are the full embeddings.
\end{theorem}

Our main result generalizes Chen-Liu-Yang's results on a triangulated category and is new for an exact category case. In particular, applying our main results to complex categories, we reobtain the Krause's recollement. In fact, without considering homotopy categories, we reprove the existence of Krause's recollement by our main resluts.

This article is organised as follows: In Section 2, we review some basic concepts and results
concerning extriangulated categories. In Section 3, we show our main results. In Section 4, we give some applications to illustrate our main results.

\section{Preliminaries}
Let us briefly recall some definitions and basic properties of extriangulated categories from \cite{NP}.
We omit some details here, but the reader can find
them in \cite{NP}.

Let $\mathcal{C}$ be an additive category equipped with an additive bifunctor
$$\mathbb{E}: \mathcal{C}^{\rm op}\times \mathcal{C}\rightarrow {\rm Ab},$$
where ${\rm Ab}$ is the category of abelian groups. For any objects $A, C\in\mathcal{C}$, an element $\delta\in \mathbb{E}(C,A)$ is called an $\mathbb{E}$-extension.
Let $\mathfrak{s}$ be a correspondence which associates an equivalence class $$\mathfrak{s}(\delta)=\xymatrix@C=0.8cm{[A\ar[r]^x
 &B\ar[r]^y&C]}$$ to any $\mathbb{E}$-extension $\delta\in\mathbb{E}(C, A)$. This $\mathfrak{s}$ is called a {\it realization} of $\mathbb{E}$, if it makes the diagrams in \cite[Definition 2.9]{NP} commutative.
 A triplet $(\mathcal{C}, \mathbb{E}, \mathfrak{s})$ is called an {\it extriangulated category} if it satisfies the following conditions.
\begin{enumerate}
\item[\rm (1)] $\mathbb{E}\colon\mathcal{C}^{\rm op}\times \mathcal{C}\rightarrow \rm{Ab}$ is an additive bifunctor.

\item[\rm (2)] $\mathfrak{s}$ is an additive realization of $\mathbb{E}$.

\item[\rm (3)] $\mathbb{E}$ and $\mathfrak{s}$  satisfy the compatibility
conditions in \cite[Definition 2.12]{NP}.
 \end{enumerate}

\begin{remark}
Note that both exact categories and triangulated categories are extriangulated categories, see \cite[Example 2.13]{NP} and extension closed subcategories of triangulated categories are
again extriangulated, see \cite[Remark 2.18]{NP}. Moreover, there exist extriangulated categories which
are neither exact categories nor triangulated categories, see \cite[Proposition 3.30]{NP} and \cite[Example 4.14]{ZZ}.
\end{remark}

We will use the following terminology.
\begin{definition}{\cite{NP}}
Let $(\C,\E,\s)$ be an extriangulated category.
\begin{itemize}
\item[(1)] A sequence $A\xrightarrow{~x~}B\xrightarrow{~y~}C$ is called a {\it conflation} if it realizes some $\E$-extension $\del\in\E(C,A)$.
    In this case, $x$ is called an {\it inflation} and $y$ is called a {\it deflation}.

\item[(2)] If a conflation  $A\xrightarrow{~x~}B\xrightarrow{~y~}C$ realizes $\delta\in\mathbb{E}(C,A)$, we call the pair $( A\xrightarrow{~x~}B\xrightarrow{~y~}C,\delta)$ an {\it $\E$-triangle}, and write it in the following way.
$$A\overset{x}{\longrightarrow}B\overset{y}{\longrightarrow}C\overset{\delta}{\dashrightarrow}$$
We usually do not write this $``\delta"$ if it is not used in the argument.

\item[(3)] Let $A\overset{x}{\longrightarrow}B\overset{y}{\longrightarrow}C\overset{\delta}{\dashrightarrow}$ and $A^{\prime}\overset{x^{\prime}}{\longrightarrow}B^{\prime}\overset{y^{\prime}}{\longrightarrow}C^{\prime}\overset{\delta^{\prime}}{\dashrightarrow}$ be any pair of $\E$-triangles. If a triplet $(a,b,c)$ realizes $(a,c)\colon\delta\to\delta^{\prime}$, then we write it as
$$\xymatrix{
A \ar[r]^x \ar[d]^a & B\ar[r]^y \ar[d]^{b} & C\ar@{-->}[r]^{\del}\ar[d]^c&\\
A'\ar[r]^{x'} & B' \ar[r]^{y'} & C'\ar@{-->}[r]^{\del'} &}$$
and call $(a,b,c)$ a {\it morphism of $\E$-triangles}.

\item[(4)] An object $P\in\C$ is called {\it projective} if
for any $\E$-triangle $A\overset{x}{\longrightarrow}B\overset{y}{\longrightarrow}C\overset{\delta}{\dashrightarrow}$ and any morphism $c\in\C(P,C)$, there exists $b\in\C(P,B)$ satisfying $yb=c$.
We denote the subcategory of projective objects by $\cal P\subseteq\C$. Dually, the subcategory of injective objects is denoted by $\cal I\subseteq\C$.

\item[(5)] We say that $\C$ {\it has enough projective objects} if
for any object $C\in\C$, there exists an $\E$-triangle
$A\overset{x}{\longrightarrow}P\overset{y}{\longrightarrow}C\overset{\delta}{\dashrightarrow}$
satisfying $P\in\cal P$. We can define the notion of having enough injectives dually.

\end{itemize}
\end{definition}

Assume that $(\C, \E, \s)$ is an extriangulated category. By Yoneda's lemma, any $\E$-extension
$\del\in\E(C,A)$ induces natural transformations
$$\del_{\sharp}\colon\C(-,C)\Rightarrow\E(-,A)\ \ \textrm{and} \ \ \del^{\sharp}\colon \C(A,-)\Rightarrow\E(C,-).$$
For any $X\in\C$, these $(\del_{\sharp})_X $ and $\del_{X}^{\sharp}$ are given as follows:

(1) $(\del_{\sharp})_X\colon\C(X,C)\to \E(X,A); f\mapsto f^{\ast}\del$.

(2) $\del_{X}^{\sharp}\colon \C(A,X)\to \E(C,X); g\mapsto g_{\ast}\del$.

\begin{lemma}\label{lem}
Let $(\C, \E, \s)$ be an extriangulated category, and $$\xymatrix{A\ar[r]^{x}&B\ar[r]^{y}&C\ar@{-->}[r]^{\delta}&}$$
an $\E$-triangle. Then we have the following long exact sequence:
$$\C(-, A)\xrightarrow{\C(-,x)}\C(-, B)\xrightarrow{\C(-,y)}\C(-, C)\xrightarrow{\delta^{\sharp}_-}
\E(-, A)\xrightarrow{\E(-,x)}\E(-, B)\xrightarrow{\E(-,y)}\E(-, C);
$$
$$\C(C,-)\xrightarrow{\C(y,-)}\C(B,-)\xrightarrow{\C(x,-)}\C(A,-)\xrightarrow{\delta_{\sharp}^-}
\E(C,-)\xrightarrow{\E(y,-)}\E(B,-)\xrightarrow{\E(x,-)}\E(A,-)
.$$
\end{lemma}

\proof This follows from Proposition 3.3 and Proposition 3.11 in \cite{NP}. \qed

Now we recall higher extension groups from \cite{LN}.

Assume that $(\C, \E, \s)$ is an extriangulated category with enough projectives and enough injectives. For two subcategories $\X,\Y$ of $\C$,  $\cone(\X,\Y)$ is defined to be the subcategory of $\C$, consisting of objects $C$ which admits an $\E$-triangle
$\xymatrix{X\ar[r]^{x}&Y\ar[r]^{y}&C\ar@{-->}[r]^{\delta}&}$
where $X\in\X$ and $Y\in\Y$. Dually we can define $\cocone(\X,\Y)$.

For a subcategory $\C_1\subseteq\C$, put $\Omega^{0}\C_1=\C_1$, and define $\Omega^i\C_1$ for $i>0$ inductively by
\[ \Omega^i\C_1=\Omega(\Omega^{i-1}\C_1)=\cocone(\P,\Omega^{i-1}\C_1). \]
We call $\Omega^{i}\C_1$ the {\it $i$-th syzygy} of $\C_1$. Dually we define the {\it $i$-th cosyzygy} $\Sigma^i\C_1$ by $\Sigma^0\C_1=\C_1$ and $\Sigma^i\C_1=\cone(\Sigma^{i-1}\C_1,\mathcal{I})$ for $i>0$.

Let $X$ be any object in $\C$. It admits an $\E$-triangle
\[ X\to I^0\to \Sigma X \overset{}{\dashrightarrow}\quad (\text{ resp. } \Omega X\to P_0\to X \overset{}{\dashrightarrow}), \]
where $I^0\in \mathcal I$ (resp. $P_0\in \mathcal P$). We can obtain $\E$-triangles
\[ \Sigma^iX\to I^i\to \Sigma^{i+1}X \overset{}{\dashrightarrow} (\text{ resp. } \Omega^{i+1}\to P_i\to \Omega^iX \overset{}{\dashrightarrow}), \]
for $i>0$ recursively.

Liu and Nakaoka \cite{LN} defined higher extension groups as
$$\E^{i+1}(X,Y):=\E(X,\Sigma^iY)\cong\E(\Omega^iX,Y)$$ for $i\geq 0$, and proved the following.

\begin{lemma}\label{lem1}\emph{\cite[Proposition 5.2]{LN}}
Let $(\C, \E, \s)$ be an extriangulated category with enough projectives and enough injectives, and $$\xymatrix{A\ar[r]^{x}&B\ar[r]^{y}&C\ar@{-->}[r]^{\delta}&}$$
an $\E$-triangle. Then we have the following long exact sequence:
$$\cdots\xrightarrow{}\E^i(-, A)\xrightarrow{}\E^i(-, B)\xrightarrow{}\E^i(-, C)\xrightarrow{}
\E^{i+1}(-, A)\xrightarrow{}\E^{i+1}(-, B)\xrightarrow{}\E^{i+1}(-, C)\xrightarrow{}\cdots~(i\geq1);
$$
$$\cdots\xrightarrow{}\E^i(C,-)\xrightarrow{}\E^i(B,-)\xrightarrow{}\E^i(A,-)\xrightarrow{}
\E^{i+1}(C,-)\xrightarrow{}\E^{i+1}(B,-)\xrightarrow{}\E^{i+1}(A,-)\xrightarrow{}\cdots~(i\geq1)
.$$
\end{lemma}

\begin{remark}\label{lem1}\cite[Lemma 2.14]{ZhZ}
Let $(\C, \E, \s)$ be an extriangulated category with enough projectives and enough injectives.
Then

(a) $P$ is projective object in $\C$ if and only if $\E^i(P,C)$=0, for any $C\in\C$ and $i\geq1$.

(b) $I$ is injective object in $\C$ if and only if $\E^i(C,I)$=0, for any $C\in\C$ and $i\geq1$.
\end{remark}

Let $\C$ be an additive category.
For two objects $A,B$ in $\X$ denote by $\X(A,B)$ the subgroup of $\Hom_{\C}(A,B)$ consisting of those morphisms which factor through an object in $\X$. Denote by $\C/\X$ the \emph{quotient category} of $\C$ modulo $\X$: the objects are the same as the ones in $\C$, for two objects $A$ and $B$ the Hom space is given by the quotient group $\Hom_{\C}(A,B)/\X(A,B)$.
Note that the quotient category $\C/\X$ is an additive category. We denote $\overline{f}$ the image of $f\colon A\to B$ of $\C$ in $\C/\X$.

\begin{remark}\label{y1}
If $\X$ is closed under direct summands, for any $C\in\C$ we have
$C\cong0$ in $\C/\X$ if and only if $C\in\X$.
\end{remark}

\begin{definition}\cite[Definition 4.1]{NP}
Assume that $(\C,\E,\s)$ is an extriangulated category.
Let $\U$ and $\V$ be two subcategories of $\C$. We call $(\U,\V)$ a \emph{cotorsion pair} if it satisfies the following conditions:
\begin{itemize}
\setlength{\itemsep}{0.5pt}
\item[{\rm (a)}] $\E(\U,\V)=0$.

\item[{\rm (b)}] For any object $C\in \C$, there are two $\E$-triangles
$$
V_C\rightarrow U_C\rightarrow C\overset{}{\dashrightarrow}~~\textrm{and}~~
C\rightarrow V^C\rightarrow U^C\overset{}{\dashrightarrow}
$$
satisfying $U_C,U^C\in \U$ and $V_C,V^C\in \V$.
\end{itemize}
\end{definition}

By definition of a cotorsion pair, we can immediately conclude:

\begin{remark}Assume that $(\C,\E,\s)$ is an extriangulated category.
 Let $(\U,\V)$ be a cotorsion pair on $\C$.  Then

 (a) $C$ belongs to $\U$ if and only if $\E(C,V)=0$;

(b) $C$ belongs to $\V$ if and only if $\E(\U,C)=0$;

(c)  $\U$ and $\V$ are closed under direct sums,
direct summands and extensions.
\end{remark}

\begin{definition}\cite[Definition 4.12]{NP}
Assume that $(\C,\E,\s)$ is an extriangulated category.
Let $(\S,\T)$ and $(\U,\V)$ be cotorsion pairs on $\C$. Then the pair
$TCP:=((\S,\T),(\U,\V))$ is called a \emph{twin cotorsion pair} if it satisfies
$\E(\S,\V)=0$. Note that this condition is equivalent to $\S\subseteq\U$, and also
to $\V\subseteq\T$. If it moreover it satisfies
$\S\cap\T=\U\cap\V$, then $TCP$ is called a \emph{concentric} twin cotorsion pair.
\end{definition}

\begin{lemma}\label{y5}
Assume that $(\C,\E,\s)$ is an extriangulated category.
 Let $(\U,\V)$ be a cotorsion pair on $\C$ and $\W:=\U\cap\V$. Then $(\C/\W)(\U/\W,\V/\W)=0$.
\end{lemma}

\proof Let $f\in\C(U,V)$ be any morphism, where $U\in\U$ and $V\in\V$.
By definition of a cotorsion pair, there exists an $\E$-triangle
$$V'\overset{x}{\longrightarrow} U'\overset{y}{\longrightarrow} V\overset{\del}\dashrightarrow$$
where $V'\in\V$ and $U'\in\U$.
Since $\V$ is closed under extensions, we have $U'\in\V$ and then $U'\in\U\cap\V=\W$.
Applying the functor $\C(U,-)$ to the $\E$-triangle $V'\overset{x}{\longrightarrow} U'\overset{y}{\longrightarrow} V\overset{\del}\dashrightarrow$, by Lemma \ref{lem}, we have the following exact sequence:
$$\C(U,U')\xrightarrow{\C(U,~y)}\C(U,V)\xrightarrow{}\E(U,V')=0.$$
It follows that there exists a morphism $a\colon U\to U'$ such that
$f=y\circ a$. Since $U'\in\W$, this means $\overline{f}=0$. \qed

\begin{lemma}\label{y6}
Assume that $(\C,\E,\s)$ is an extriangulated category.
Let $(\U,\V)$ be a cotorsion pair on $\C$ and $\W:=\U\cap\V$, and let
$f\colon A\rightarrow B$ be any morphism in $\mathcal{C}$.
\begin{enumerate}
\item[\rm (1)] Let
\begin{eqnarray*}
V_A\xrightarrow{}U_A\overset{u_A}{\longrightarrow}A\dashrightarrow\\
V_B\xrightarrow{}U_B\overset{u_B}{\longrightarrow}B\dashrightarrow
\end{eqnarray*}
be any $\E$-triangles satisfying $U_A,U_B\in\mathcal{U}$ and $V_A,V_B\in\mathcal{V}$.
Then there exists a morphism $f_U\in\mathcal{C}(U_A,U_B)$ such that
\begin{eqnarray*}
f\circ u_A=u_B\circ f_U.\\
\xy
(-8,8)*+{U_A}="0";
(-8,-8)*+{U_B}="2";
(8,8)*+{A}="4";
(8,-8)*+{B}="6";
{\ar_{f_U} "0";"2"};
{\ar_{u_B} "2";"6"};
{\ar^{u_A} "0";"4"};
{\ar^{f} "4";"6"};
{\ar@{}|\circlearrowright "0";"6"};
\endxy
\end{eqnarray*}
Moreover, $f_U$ with this property is unique in $(\C/\W)(U_A,U_B)$.
\item[\rm (2)] Dually, for any $\E$-triangles
\begin{eqnarray*}
 A\rightarrow V^{\prime}_A\rightarrow U^{\prime}_A\dashrightarrow\\
B\rightarrow V^{\prime}_B\rightarrow U^{\prime}_B\dashrightarrow
\end{eqnarray*}
with $U^{\prime}_A,U^{\prime}_B\in\mathcal{U}$ and $V^{\prime}_A,V^{\prime}_B\in\mathcal{V}$, there exists a morphism $f^{\prime}_V\in\mathcal{C}(V^{\prime}_A,V^{\prime}_B)$ compatible with $f$, uniquely up to $\mathcal{W}$.
\end{enumerate}
\end{lemma}

\proof We only show {\rm (1)}. Existence immediately follows from $\E(U_A,V_B)=0$. Moreover if $f_U^1$ and $f_U^2$ in $\mathcal{C}(U_A,U_B)$ satisfies
\[ f_U^1\circ u_A=u_B\circ f_U=f^2_U\circ u_A, \]
then by $(f_U^1-f_U^2)\circ u_A=0$, there exists $w\in\mathcal{C}(U_A,V_B)$ such that $f_U^1-f_U^2$ factors through $w$.
\[
\xy
(-24,-8)*+{V_B}="0";
(-8,-8)*+{U_B}="2";
(8,-8)*+{B}="4";
(24,-8)*+{}="6";
(-8,8)*+{U_A}="8";
(8,8)*+{A}="10";
(-18,4)*+{}="12";
{\ar "0";"2"};
{\ar "2";"4"};
{\ar@{-->} "4";"6"};
{\ar^{f_U^1-f_U^2} "8";"2"};
{\ar^{f} "10";"4"};
{\ar@{-->}_{w} "8";"0"};
{\ar^{u_A} "8";"10"};
{\ar@{}|\circlearrowright "2";"12"};
\endxy
\]
By Lemma \ref{y5} we have $\overline{w}=0$, and then $\overline{f_U^1}=\overline{f_U^2}$.\qed

\begin{lemma}\label{y7}
Assume that $(\C,\E,\s)$ is an extriangulated category.
Let $(\U,\V)$ be a cotorsion pair on $\C$, $\W:=\U\cap\V$ and $C$ be any object in $\mathcal{C}$.
\begin{enumerate}
\item[\rm (1)] For any $\E$-triangles
\begin{eqnarray*}
V\xrightarrow{}U\overset{u}{\longrightarrow}C\dashrightarrow\\
V'\xrightarrow{}U^{\prime}\overset{u^{\prime}}{\longrightarrow}C\dashrightarrow
\end{eqnarray*}
satisfying $U,U^{\prime}\in\mathcal{U}$ and $V,V^{\prime}\in\mathcal{V}$, there exists a morphism $s\in\mathcal{C}(U,U^{\prime})$ compatible with $u$ and $u^{\prime}$, such that $\overline{s}$ is an isomorphism.
\[
\xy
(0,8)*+{U}="0";
(0,-8)*+{U^{\prime}}="2";
{\ar_{\rotatebox{90}{$\cong$}}^{\overline{s}} "0";"2"};
\endxy
\qquad
\xy
(-4,8)*+{U}="0";
(-4,-8)*+{U^{\prime}}="2";
(6,0)*+{C}="4";
(-8,0)*+{}="6";
{\ar_{s} "0";"2"};
{\ar^{u} "0";"4"};
{\ar_{u^{\prime}} "2";"4"};
{\ar@{}|\circlearrowright "4";"6"};
\endxy
\]
\item[\rm (2)] Dually, those $V$ appearing in $\E$-triangles
\[C\rightarrow V\rightarrow U\dashrightarrow\quad \emph{where}~ U\in\mathcal{U},V\in\mathcal{V}\]
are isomorphic in $\C/\W$.
\end{enumerate}
\end{lemma}

\proof This immediately follows from Lemma \ref{y6}.  \qed

\section{Recollements of additive categories}
In this section, we will prove our main results, and need to do some preparations as follows.

\begin{definition}\cite[Definition 2.1]{WL} and \cite[Definition 3.1]{CLY}
Let $\A'\xrightarrow{~F~}\A\xrightarrow{~G~}\A''$
be a sequence of additive functors between
additive categories. We say it is a \emph{localization sequence} if the following conditions
hold:
\begin{itemize}
\item[\rm(L1)] The functor $F$ is fully faithful and has a right adjoint $F_{\del}$.
\item[\rm (L2)] The functor $G$ has a fully faithful right adjoint $G_{\del}$.
\item[(L3)] There exists an equality of additive subcategories
$\Im F=\Ker G$, where $$\Im F=\{A\in\A~|~A\cong F(X)~\textrm{for some}~ X\in\A'\}
~~\textrm{and }~\Ker G=\{A\in\A~|~ G(A)=0~ \textrm{in}~\A'' \}.$$
\end{itemize}
\emph{Colocalization sequence} of additive categories is defined dually.
A sequence of
additive categories $\A'\xrightarrow{~F~}\A\xrightarrow{~G~}\A''$ is called a \emph{recollement} if it is both a localization
sequence and a colocalization sequence.
\end{definition}

From now on, we assume that $(\C, \E, \s)$ is an extriangulated category with enough projectives and enough injectives.

\begin{definition}\cite[Definition 3.1.1]{B}
Let $\A$ and $\B$ be additive categories and $F\colon \A\to\B$ a functor.
For any $B\in\B$, a \emph{reflection} of $B$ along $F$  is a pair $(Q_B,\eta_B)$ of $Q_B\in\A$ and
$\eta_B\in\B(B,F(Q_B))$, satisfying the universality that for any $A\in\A$ and any $b\in\B(B,F(A))$, there exists a unique morphism $a\in\A(Q_B ,A)$ such that $F(a)\circ\eta_B=b$.
$$\xymatrix{
  B \ar[rr]^{\eta_B} \ar[dr]_{b}
                &  &    F(Q_B) \ar[dl]^{F(a)}    \\
                & F(A)                 }$$
A \emph{coreflection} is defined dually.
\end{definition}

\begin{lemma}\label{c4}
Let $((\S,\T ),(\U,\V))$ be a concentric twin cotorsion pair on $\C$ with
$\W= \S\cap\T$. Then the inclusion $\V/\W\hookrightarrow\T/\W$ admits an additive left adjoint
$F_{\lambda}\colon\T/\W\to\V/\W$. Indeed, for any $T\in\T$ and $\E$-triangle
$$T\overset{v_T}{\longrightarrow} V_T\overset{}{\longrightarrow} U_T\overset{}\dashrightarrow$$
with $U_T\in\U$ and $V_T\in\V, \overline{v_T}\colon T\to V_T$ gives a reflection of $T$ along the inclusion
$\V/\W\hookrightarrow\T/\W$, where $F_{\lambda}(T)=V_T$.
\end{lemma}

\proof  It is enough to prove that for any $T\in\T, V\in\V$ and $f\in\C(T,V )$, there exists
a morphism $g\in\C(V_T,V)$ such that $f=g\circ v_T$, uniquely in $(\C/\W)(V_T,V)$.

Applying the functor $\C(-,V)$ to the $\E$-triangle $T\overset{v_T}{\longrightarrow} V_T\overset{}{\longrightarrow} U_T\overset{}\dashrightarrow$, we have the following exact sequence:
$$\C(V_T,V)\xrightarrow{\C(v_T,V)}\C(T,V)\xrightarrow{}\E(U_T,V)=0.$$
Thus there exists a morphism $g\in\C(V_T,V)$ such that $f=g\circ v_T$ and then $\overline{f}=\overline{g}\circ \overline{v_T}$.

To prove uniqueness now, suppose that there exists a morphism $g\in\C(V_T,V)$ such that
$\overline{f} =\overline{g'}\circ \overline{v_T}$. Then $(\overline{g}-\overline{g'})\circ v_T=0$, that is, $(g-g')\circ v_T$ factors through some $W_0\in\W$.
Let $(g-g')\circ v_T=b\circ a$ with $a\colon T\to W_0$ and $b\colon W_0\to V$.

Applying the functor $\C(-,W_0)$ to the $\E$-triangle $T\overset{v_T}{\longrightarrow} V_T\overset{}{\longrightarrow} U_T\overset{}\dashrightarrow$, we have the following exact sequence:
$$\C(V_T,W_0)\xrightarrow{\C(v_T,W_0)}\C(T,W_0)\xrightarrow{}\E(U_T,W_0)=0.$$
Thus there exists a morphism $c\in\C(V_T,W_0)$ such that $a=c\circ v_T$ and then
$(g-g'-b\circ c)\circ v_T=0$.
By Lemma \ref{lem},  there exists a morphism $d\in\C(U_T,V)$ such that
$(g-g')-b\circ c = d\circ e$. It follows that $\overline{g}-\overline{g'}=\overline{d}\circ \overline{e}$.
By Lemma \ref{y5}, we have $(\C/\W)(U_T,V)=0$ implies $\overline{g}=\overline{g'}$.
Therefore, $v_T\colon T\to V_T$ gives a reflection of $T$ along the inclusion
$\V/\W\hookrightarrow\T/\W$. Namely, $F_{\lambda}\colon \T/\W\to \V/\W$ is left adjoint to the inclusion $\V/\W\to\T/\W$.  \qed

\begin{lemma}\label{c5}
Let $((\S,\T ),(\U,\V))$ be a concentric twin cotorsion pair on $\C$ with
$\W= \S\cap\T$. If $\Sigma\V\subseteq\V$, then there exists the following colocalization sequence of
additive categories:
$$\xymatrixcolsep{5pc}\xymatrix{
&\V/\W\ar[r]^{F}
&\T/\W\ar@/_1.5pc/[l]_{F_{\lambda}}\ar[r]^{G}
&(\T\cap\U)/\W\ar@/_1.5pc/[l]_{G_{\lambda}}},$$
where $F$ and $G_\lambda$ are the full embeddings.
\end{lemma}

\proof
Assume that $F\colon \V/\W\to \T /\W$ is the full embedding.
We define $G\colon\T/\W\to(\T\cap\U)/\W$ in the following way.
For any $T\in\T, G(T)=Z_T$ appearing in an $\E$-triangle
$V_T\overset{}{\longrightarrow} Z_T\overset{}{\longrightarrow} T\overset{}\dashrightarrow$
 with $Z_T\in \T\cap\U$ and $V_T\in\V$. By
\cite[Definition 3.13]{ZW}, we know that $G$ is an additive functor and a right adjoint of the inclusion $(\T\cap\U)/\W\to\T/\W$, denoted by $G_{\lambda}$. We claim that
 $\Im F = \Ker G$. It is easy to see that $\Im F = \V$ since $F$ is the
full embedding. For any $V\in\V$, there exists an $\E$-triangle
 $$V_0\overset{}{\longrightarrow} W_0\overset{}{\longrightarrow} V\overset{}\dashrightarrow$$
  with $W_0\in\W$ and $V_0\in\V$. Hence, $G(V)=W_0$ in
$(\T\cap\U)/\W$. So we know that $V\in\Ker G$ by Remark \ref{y1}. Then $\Im F\subseteq\Ker G$. On the other hand,
for any $T\in\Ker G$, there exists an $\E$-triangle
$$V_T\overset{}{\longrightarrow} Z_T\overset{}{\longrightarrow} T\overset{}\dashrightarrow$$
with $Z_T\in\W$ and $V_T\in\V$.
Since $\Sigma\V\subseteq\V$, we have $\E(U,\Sigma V_T)=0$.
Applying the functor $\C(\U,-)$ to the $\E$-triangle $T\overset{v_T}{\longrightarrow} V_T\overset{}{\longrightarrow} U_T\overset{}\dashrightarrow$, by Lemma \ref{lem1}, we obtain the following exact sequence:
$$\E(\U,Z_T)=0\xrightarrow{~~}\E(\U,T)\xrightarrow{~~}\E(\U,\Sigma V_T)=0.$$
It follows that $\E(\U,T)=0$ and then $T\in\V$.
Thus $\Im F = \Ker G$. Let $F_{\lambda}$ be as in Lemma \ref{c4}.
Then $(F_{\lambda},F)$
is an adjoint pair. This completes the proof.  \qed

\begin{lemma}\label{c6}
Let $(\U,\V)$ be a cotorsion pair and $((\S,\T ),(\X,\Y))$ a concentric twin cotorsion pair on $\C$ with $\W= \S\cap\T$. If $\T\cap\X=\V$,  then the inclusion $\V/\W\to\T/\W$
admits an additive right adjoint functor $F_{\del}\colon\T/\W\to\V/\W$. Indeed, for any
$T\in\T$
and $\E$-triangle
$$Y_T\overset{}{\longrightarrow} V_T\overset{x_T}{\longrightarrow} T\overset{}\dashrightarrow$$
with $Y_T\in\Y$ and $V_T\in\V, \overline{x_T}\colon V_T\to T$ gives a coreflection of $T$ along the inclusion
$\V/\W\hookrightarrow\T /\W$, where $F_{\del}(T)=V_T$.
\end{lemma}

\proof For any $T\in\T$, there exists an $\E$-triangle
$$Y_T\overset{}{\longrightarrow} V_T\overset{x_T}{\longrightarrow} T\overset{}\dashrightarrow$$
with $Y_T\in\Y$ and $V_T\in\X$. Since $\Y\subseteq\T$ and $\T$ is closed under extensions, we have
$V_T\in\T$.
Thus $V_T\in\X\cap\T = \V$. The remaining proof is similar to Lemma \ref{c4}.  \qed

\begin{lemma}\label{c7}
Let $((\S,\T ),(\X,\Y))$ a concentric twin cotorsion pair on $\C$ with $\W= \S\cap\T$.
Then the inclusion $\Y/\W\hookrightarrow\T/\W$ admits an additive left adjoint $G'\colon
\T/\W\to\Y/\W$. Indeed, for any $T\in\T$ and $\E$-triangle
$$T\overset{y_T}{\longrightarrow} Y_T\overset{}{\longrightarrow}X_T\overset{}\dashrightarrow$$
with $Y_T\in\Y$ and $X_T\in\X, \overline{y_{T}}\colon T\to Y_T$ gives a reflection of $T$ along the inclusion
$\Y/\W\hookrightarrow \T /\W$, where $G'(T) = Y_T$.
\end{lemma}

\proof  The proof is similar to Lemma \ref{c4}.  \qed

\begin{lemma}\label{c8}
Let $(\U,\V)$ be a cotorsion pair and $((\S,\T),(\X,\Y))$ a concentric twin
cotorsion pair on $\C$ with $\W=\S\cap\T$. If $\T\cap\X =\V$ and $\Sigma\Y\subseteq\Y$, then there exists the following localization sequence of additive categories:
$$\xymatrixcolsep{5pc}\xymatrix{
&\V/\W\ar[r]^{F}
&\T/\W \ar@/^1.5pc/[l]^{F_{\del}} \ar[r]^{G'}
&\Y/\W\ar@/^1.5pc/[l]^{G'_{\del}}},$$
where $F$ and $G'_{\del}$ are the full embeddings.
\end{lemma}

\proof Suppose that $F\colon\V/\W\to \T /\W$ and $G'_{\del}\colon\Y/\W\to\T/\W$ are the full embeddings. Let $F_{\del}$ and $G'$ be as in Lemma \ref{c6} and Lemma \ref{c7}, respectively. Then
$(F,F_{\del})$ and $(G',G'_{\del})$ are adjoint pairs. Next we claim that $\Im F = \Ker G'$.
Obviously, $\Im F = \V$. For any $V\in\V$, there exists an $\E$-triangle
$$V\overset{}{\longrightarrow} Y_V\overset{}{\longrightarrow}X_V\overset{}\dashrightarrow$$
with $Y_V\in\Y$ and $X_V\in\X$. Since $\T\cap\X = \V$ and
$\X$ is closed under extensions, $Y_V\in\Y\cap\X = \W$. Since $G'(V )=Y_V , V \in\Ker G'$ by
Remark \ref{y1}. Then $\Im F\subseteq\Ker G'$. On the other hand, for any $T\in\Ker G', G'(T) =Y_T$
appearing in an $\E$-triangle
$$T\overset{}{\longrightarrow} Y_T\overset{}{\longrightarrow}X_T\overset{}\dashrightarrow$$
$Y_T\in\Y$ and
$X_T\in\X$. Then $Y_T\in\W$.

Applying the functor $\C(-,\Y)$ to the $\E$-triangle $T\overset{}{\longrightarrow} Y_T\overset{}{\longrightarrow}X_T\overset{}\dashrightarrow$, by Lemma \ref{lem1}, we have the following exact sequence:
$$\E(Y_T,\Y)=0\xrightarrow{~~}\E(T,\Y)\xrightarrow{~~}\E(X_T,\Sigma\Y)=0.$$
It follows that $\E(T,\Y)=0$ and then $T\in\X$.
 Thus $T\in\T\cap\X = \V$, namely, $T\in\Im F$.
So $\Im F = \Ker G'$. This completes the proof.  \qed

\begin{theorem}\label{main1}
Let $(\S,\T), (\U,\V)$ and $(\X,\Y)$ be cotorsion pairs on $\C$ with
$\W=\S\cap\T = \U\cap\V =\X\cap\Y$. If $\X\cap\T = \V, \Y\subseteq\T, \Sigma\V\subseteq\V$ and
 $\Sigma\Y\subseteq\Y$, then there
are two recollements of additive categories as follows:
$$\xymatrixcolsep{5pc}\xymatrix{
&\V/\W\ar[r]^{F}
&\T/\W\ar@/_1.5pc/[l]_{F_{\lambda}} \ar@/^1.5pc/[l]^{F_{\del}} \ar[r]^{G}
&(\T\cap\U)/\W\ar@/^1.5pc/[l]^{G_{\del}}\ar@/_1.5pc/[l]_{G_{\lambda}}},$$
$$\xymatrixcolsep{5pc}\xymatrix{
&\V/\W\ar[r]^{F}
&\T/\W\ar@/_1.5pc/[l]_{F_{\lambda}} \ar@/^1.5pc/[l]^{F_{\del}} \ar[r]^{G'}
&\Y/\W\ar@/^1.5pc/[l]^{G'_{\del}}\ar@/_1.5pc/[l]_{G'_{\lambda}}},$$
where $F, G_{\lambda}$ and $G_{\del}$ are the full embeddings.
\end{theorem}

\proof Define $J\colon(\T\cap\U)/\W\to\Y/\W$ in the following way.
For any $M\in\T\cap\U, J(M)=Y_M$ appearing in an $\E$-triangle
$$M\overset{y_M}{\longrightarrow} Y_M\overset{}{\longrightarrow}X_M\overset{}\dashrightarrow$$
with $Y_M\in\Y$ and $X_M\in\X$.
Applying the functor $\C(-,Y_{M'})$ to the above $\E$-triangle, we have the following exact sequence:
$$\C(Y_M,Y_{M'})\xrightarrow{\C(y_M,W_0)}\C(M,Y_{M'})\xrightarrow{}\E(X_M,Y_{M'})=0.$$
Thus for any $f\in(\T\cap\U)/\W)(M,M')$, there exists a morphism $f_M\in\C(Y_M,Y_{M'})$ such that $y_{M'}\circ f=f_M\circ y_M$.
Hence we have the following commutative diagram
$$\xymatrix{M\ar[r]^{y_M}\ar[d]^f &Y_M\ar[r]\ar[d]^{f_M} &X_M\ar@{-->}[r]\ar@{-->}[d]&\\
M'\ar[r]^{y_{M'}} & Y_{M'}\ar[r]& X_{M'}\ar@{-->}[r]&}$$
of $\E$-triangles.
Define $J(f)=f_M$. By Lemma \ref{y6} and Lemma \ref{y7}, $J$ is well defined and is
an additive functor.

Define $K\colon\Y/\W\to (\T\cap\U)/\W$ in the following way. For any
$Y\in\Y, K(Y)= U_Y$ appearing in an $\E$-triangle
$$V_Y\overset{}{\longrightarrow} U_Y\overset{u_Y}{\longrightarrow}Y\overset{}\dashrightarrow$$
with $V_Y\in\V$ and $U_Y\in\U$. Since $\Y\subseteq\T, \V\subseteq\T$ and $\T$ is closed under extensions,
$U_Y\in\T\cap\U$. For any $g\in(\Y/\W)(Y,Y')$, since $\E(U_Y ,V_{Y'}) = 0$, there exists
a morphism $g_Y\in\C(U_Y,U_{Y'})$ such that $u_{Y'}\circ g_Y = g\circ u_Y$.
Hence we have the following commutative diagram
$$\xymatrix{V_Y\ar[r]\ar@{-->}[d]&U_Y\ar[r]^{u_Y}\ar[d]^{g_Y} &Y\ar@{-->}[r]\ar[d]^{g}&\\
V_{Y'}\ar[r] & U_{Y'}\ar[r]^{u_{Y'}}& Y'\ar@{-->}[r]&}$$
of $\E$-triangles.
Define $K(g)=g_Y$. By Lemma \ref{y6} and Lemma \ref{y7}, $K$ is well defined and is
an additive functor. Now, we have the following diagram:
$$\xymatrix{\mathcal{V}/\mathcal{W}\ar@<-0.8ex>[rr]_-{F}\ar@<0.5ex>[dd]^{id}&&\ar@<-0.8ex>[ll]_-{F_{\lambda}}\mathcal{T}/\mathcal{W}\ar@<-0.8ex>[rr]_-{G}\ar@<0.5ex>[dd]^{id}&&\ar@<-0.8ex>[ll]_-{G_{\lambda}}(\mathcal{T}\bigcap\mathcal{U})/\mathcal{W} \ar@<-0.5ex>[dd]_{J} \\
&&&&\\
\ar@<0.5ex>[uu]^{id}\mathcal{V}/\mathcal{W}\ar@<0.8ex>[rr]^-{F}&&\ar@<0.8ex>[ll]^-{F_{\rho}}\ar@<0.5ex>[uu]^{id}\mathcal{T}/\mathcal{W}\ar@<0.8ex>[rr]^-{G'}&& \ar@<0.8ex>[ll]^-{G'_{\rho}}\mathcal{Y}/\mathcal{W}\ar@<-0.5ex>[uu]_{K}            }$$
For any $Y\in\Y$, since $0\to Y\to Y\dashrightarrow$ is an $\E$-triangle, we have $G'(Y) = Y$. So
$G'$ is dense. Then there exists a natural isomorphism from the composite functor
$\T/\W\xrightarrow{G'}\Y/\W\xrightarrow{K}(\T\cap\U)/\W$ to the functor
$\Y/\W\xrightarrow{K}(\T\cap\U)/\W$. Consequently, we know that the composite functor $\Y/\W\xrightarrow{K}
(\T\cap\U)/\W\xrightarrow{G_{\lambda}}\T /\W$ is indeed left
adjoint to $\T/\W\xrightarrow{G'}\Y/\W$. Put $G'_{\lambda}=G_{\lambda}K$. Then we have a recollement of additive
categories
$$\xymatrixcolsep{5pc}\xymatrix{
&\V/\W\ar[r]^{F}
&\T/\W\ar@/_1.5pc/[l]_{F_{\lambda}} \ar@/^1.5pc/[l]^{F_{\del}} \ar[r]^{G'}
&\Y/\W\ar@/^1.5pc/[l]^{G'_{\del}}\ar@/_1.5pc/[l]_{G'_{\lambda}}}.$$
Similarly, for any $M\in\T\cap\U, G(M)=M$ since $0 \to W\to W\dashrightarrow$ is an
$\E$-triangle. So $G$ is dense. Then there exists a natural isomorphism from the
composite functor $\T/\W\xrightarrow{G}(\T\cap\U)/\W
\xrightarrow{J}\Y/\W$ to the functor $(\T\cap\U)/\W\xrightarrow{J}\Y/\W$.
Thus $(\T\cap\U)/\W\xrightarrow{J}\Y/\W\xrightarrow{G'_{\del}}
\T/\W$ is indeed right adjoint to $\T/\W\xrightarrow{G}(\T\cap\U)/\W$. Put
$G_{\del}=G'_{\del}J$. Then we obtain the desired recollement of additive categories
$$\xymatrixcolsep{5pc}\xymatrix{
&\V/\W\ar[r]^{F}
&\T/\W\ar@/_1.5pc/[l]_{F_{\lambda}} \ar@/^1.5pc/[l]^{F_{\del}} \ar[r]^{G}
&(\T\cap\U)/\W\ar@/^1.5pc/[l]^{G_{\del}}\ar@/_1.5pc/[l]_{G_{\lambda}}}.$$  \qed

\medskip
We give the dual result of Theorem \ref{main1}, but omit its proof.
\begin{theorem}\label{main2}
Let $(\S,\T), (\U,\V)$ and $(\X,\Y)$ be cotorsion pairs on $\C$ with
$\W=\S\cap\T = \U\cap\V =\X\cap\Y$. If $\Y\cap\S = \U, \X\subseteq\S, \Sigma\V\subseteq\V$ and
 $\Sigma\Y\subseteq\Y$, then there
are two recollements of additive categories as follows:
$$\xymatrixcolsep{5pc}\xymatrix{
&\U/\W\ar[r]^{I}
&\S/\W\ar@/_1.5pc/[l]_{I_{\lambda}} \ar@/^1.5pc/[l]^{I_{\del}} \ar[r]^{J}
&(\S\cap\V)/\W\ar@/^1.5pc/[l]^{J_{\del}}\ar@/_1.5pc/[l]_{J_{\lambda}}},$$
$$\xymatrixcolsep{5pc}\xymatrix{
&\U/\W\ar[r]^{I}
&\S/\W\ar@/_1.5pc/[l]_{I_{\lambda}} \ar@/^1.5pc/[l]^{I_{\del}} \ar[r]^{J'}
&\X/\W\ar@/^1.5pc/[l]^{J'_{\del}}\ar@/_1.5pc/[l]_{J'_{\lambda}}},$$
where $I, J_{\del}$ and $J'_{\lambda}$ are the full embeddings.
\end{theorem}

\begin{remark}
When we apply Theorem \ref{main1} to a triangulated category, it is just the Theorem 3.2 in \cite{CLY}.
When we apply Theorem \ref{main2} to a triangulated category, it is just the Theorem 3.9 in \cite{CLY}.
\end{remark}

\section{Applications}
We assume that $R$ is a ring with unit. We denote by Mod$R$ the category of left $R$-modules. Unless otherwise stated, all modules are left modules. In this section, we will give some applications.  Meanwhile, we will see that Proposition 4.3 and Proposition 4.6 in \cite{CLY} can be obtained directly from the cotorsion pairs on the categories of  the complexes  by our main results. \par
Let $\mathcal{A}$ be  an abelian category with enough projective and injective objects.  For a subclass $\mathcal{C}$ of $\mathcal{A}$,  we set
\begin{align*}
  ^{\perp}\mathcal{C} &=\{~M\in \mathcal{A}~|~\text{Ext}^{1}_{R}(M,N)=0~\text{for any}~N\in\mathcal{C}\} \\
\mathcal{C}^{\perp} &=\{~N\in \mathcal{A}~|~\text{Ext}^{1}_{R}(M,N)=0~\text{for any}~M\in\mathcal{C}\}.
\end{align*}\par
A pair $(\mathcal{X}, \mathcal{Y})$ of $\mathcal{A}$ is said to be a \emph{cotorsion pair} if $\mathcal{X}={^{\perp}\mathcal{Y}}$ and $\mathcal{Y}={\mathcal{X}^{\perp}}$. The cotorsion pair is called \emph{complete} if for any $M\in$$\mathcal{A}$, there exist two short exact sequences
 \begin{align*}
   0&\rightarrow Y\rightarrow X\rightarrow M\rightarrow0 \\
   0&\rightarrow M\rightarrow Y'\rightarrow X'\rightarrow0
 \end{align*}
with $X,X'\in \mathcal{X}$ and $Y,Y'\in \mathcal{Y}$. A pair of cotorsion pairs $((\mathcal{S}, \mathcal{T}),(\mathcal{U}, \mathcal{V})$ is said to be a \emph{twin cotorsion pair} if $\mathcal{S}\subseteq \mathcal{U}$, or equivalently, $\mathcal{V}\subseteq \mathcal{T}$. A twin cotorsion pair is called \emph{concentric} if $\mathcal{S}\cap\mathcal{T}=\mathcal{U}\cap\mathcal{V}$.\par
 It is well-known that  the canonical cotorsion pair $(\mathcal{P}, \text{Mod}R)$ is completed by \cite[Theorem 3.2.1]{GT} since it is generated by $R$.

Recall that $R$ is said to be a \emph{Gorenstein ring}  if  it is a left and right noetherian ring with finite injective dimension on either side.
 An $R$-module $N$ is called \emph{Gorenstein projective}  if there is a exact sequence of projective $R$-modules
 $$\textbf{P}:~~~\cdots\rightarrow P_{1}\rightarrow P_{0} \rightarrow P^{0} \rightarrow P^{1}\rightarrow\cdots $$
  with $N=\text{Ker}(P_{0} \rightarrow P^{0} )$ such that $\text{Hom}(\textbf{P},Q)$ is exact for any projective module $Q$.

Let $\mathcal{GP}(R)$  be the full subcategory of Mod$R$ consisting of all the Gorenstein projective  module. We denote by $\mathcal{P}^{<\infty}$ the class of modules admitting finite projective dimension. If $R$ is a Gorenstein ring, then $(\mathcal{GP}(R),\mathcal{P}^{<\infty})$ is a complete cotorsion pair, see \cite[Example 4.1.14]{GT}.\par
\begin{proposition} Let $R$  be a Gorenstein ring. Then the stable category $\underline{\mathcal{GP}(R)}$ is a coreflection subcategory of {\rm$\underline{\text{Mod}R}$}.
\end{proposition}
\begin{proof} It is easy to see that (($\mathcal{P}, \text{Mod}R$),($\mathcal{GP}(R),\mathcal{P}^{<\infty}$)) is a concentric twin cotorsion pair and $\Sigma \mathcal{P}^{<\infty}\subseteq \mathcal{P}^{<\infty}$ since $R$  is a Gorenstein ring.  By Lemma \ref{c5}, we obtain a  colocalization sequence of additive categories
$$\xymatrixcolsep{5pc}\xymatrix{
&\underline{\mathcal{P}^{<\infty}}\ar[r]^{F}
&\underline{\text{Mod}R}\ar@/_1.5pc/[l]_{F_{\lambda}}\ar[r]^{G}
&\underline{\mathcal{GP}(R)}\ar@/_1.5pc/[l]_{G_{\lambda}}}.$$
Thus, the inclusion functor $G_{\lambda}$ admits a right adjoint functor. That is, $\underline{\mathcal{GP}(R)}$ is a coreflection subcategory of {\rm$\underline{\text{Mod}R}$}.
\end{proof}
 We denote the class of projective modules by $\mathcal{P}$, the class of injective modules by $\mathcal{I}$, the the category of a complex of left $R$-modules by $C(\text{Mod}R)$, the class of exact complexes by $\mathcal{E}$.\par
We write a complex $X^{\bullet}\in C(\text{Mod}R)$ as
$$\cdots\longrightarrow X_{n+1}\xrightarrow{d_{n+1}}X_{n}\xrightarrow{d_{n}}X_{n-1}\longrightarrow\cdots.$$
Given $X^{\bullet}\in C(\text{Mod}R)$, the suspension of $X^{\bullet}$, denoted by $S(X^{\bullet})$, is the complex given by $S(X^{\bullet})_{n}=X_{n-1}$ and $d_{S(X^{\bullet})}=-d_{n}$. Inductively, one can define $S^{n+1}(X^{\bullet})=S(S^{n}(X^{\bullet}))$ for all $n\in\mathbb{Z}$. Given an $R$-module $X$, we denote by $\overline{X}$ the complex
$$\cdots\longrightarrow 0\longrightarrow X\overset{1}{\longrightarrow}X\longrightarrow 0\longrightarrow \cdots$$
where the two $X$'s are in the 0th and -1st place.
 \par
 Recall that a complex $P^{\bullet}\in C(\text{Mod}R)$ is said to be \emph{projective }if for any morphism $P^{\bullet}\rightarrow D^{\bullet}$ and any epimorphism $C^{\bullet}\rightarrow D^{\bullet}$, the diagram
$$\xymatrix{
                &         P^{\bullet} \ar[d] \ar@{-->}[dl] \\
  C^{\bullet} \ar[r] & D^{\bullet}             }$$
can be completed to a commutative diagram by a morphism $ P^{\bullet} \rightarrow  C^{\bullet} $. Dually, one can define the injective complex $I^{\bullet}$. It is well-known that each projective complex $P^{\bullet}$ and injective complex $I^{\bullet}$ are exact complexes. Moreover, the components $P_{n}$, $I_{n}$ of projective complex $P^{\bullet}$ and injective complex $I^{\bullet}$ are projective and injective modules in Mod$R$, respectively. However, conversely, it  may not true, see \cite[Example 1.4.5]{EJ}.\par
From \cite[Section 1.4]{EJ}, we know that $C(\text{Mod}R)$ admits enough projective and injective complexes. Then, for any complex $C^{\bullet} \in C(\text{Mod}R)$, there
exist a projective resolution and an injective resolution of $C^{\bullet}$. It means that there exist two exact
sequences of complexes
$$\cdots\longrightarrow{P^{\bullet}}_{n}\longrightarrow\cdots\longrightarrow P^{\bullet}_{1}\longrightarrow P^{\bullet}_{0}\longrightarrow C^{\bullet}\longrightarrow0$$
$$0\longrightarrow C^{\bullet}\longrightarrow I^{\bullet}_{0}\longrightarrow I^{\bullet}_{-1}\longrightarrow \cdots \longrightarrow I^{\bullet}_{-n}\longrightarrow\cdots$$
where each $P^{\bullet}_{n}$, $I^{\bullet}_{-n}$ is  projective and  injective, respectively.\par
Now, we can define the  groups $\text{Ext}^{n}_{C(\text{Mod}R)}(C^{\bullet}, D^{\bullet})$ (or simply, $\text{Ext}^{n}(C^{\bullet}, D^{\bullet})$) for any complexes $C^{\bullet}$ and $D^{\bullet}$. If exact
sequences of complexes
$$\cdots\longrightarrow P^{\bullet}_{n}\longrightarrow\cdots\longrightarrow P^{\bullet}_{1}\longrightarrow P^{\bullet}_{0}\longrightarrow C^{\bullet}\longrightarrow0$$
is the projective resolution of $C^{\bullet}$, then $\text{Ext}^{n}(C^{\bullet}, D^{\bullet})$ is defined to be the  $n$-cohomology group of the complex
$$0\rightarrow \text{Hom}(P^{\bullet}_{0},D^{\bullet})\rightarrow\text{Hom}(P^{\bullet}_{1},D^{\bullet})\rightarrow\cdots.$$
It also can be computed by the injective resolution of $D^{\bullet}$. Especially, the $\xi\in \text{Ext}^{1}(C^{\bullet}, D^{\bullet})$ can be put in bijective correspondence with the equivalence classes of short exact sequences
$$0\longrightarrow D^{\bullet}\longrightarrow E^{\bullet}\longrightarrow C^{\bullet}\longrightarrow0 $$ in $C(\text{Mod}R)$.\par
We denote by $\textbf{K}(\text{Mod}R)$ and $\textbf{D}(R)$ the corresponding homotopy category and derived category of $\text{Mod}R$.
\begin{definition}\cite{G01, G02}
Given a class of $R$-modules $\mathcal{C}$, we define the following classes of chain complexes in $C(\text{Mod}R)$
\begin{enumerate}
  \item[(1)] $dw \mathcal{C}$ is the class of all chain complexes $C^{\bullet}$ with $C_{n}\in\mathcal{C}$ for any $n\in\mathbb{Z}$.
  \item[(2)] $ex \mathcal{C}$ is the class of all exact chain complexes $C^{\bullet}$ with $C_{n}\in\mathcal{C}$ for any $n\in\mathbb{Z}$.
  \item[(3)] $\widetilde{\mathcal{C}}$ is the class of all exact chain complexes $C^{\bullet}$ with cycles $Z_{n}(C^{\bullet})\in\mathcal{C}$.
  \item[(4)] Given any cotorsion pair ($\mathcal{X}$, $\mathcal{Y}$) in Mod$R$. $dg \mathcal{X}$ is the class of all complexes of $R$-modules satisfying that each $X_{n}\in \mathcal{X}$ and $\text{Hom}_{\textbf{K}(\text{Mod}R)}(X^{\bullet},Y^{\bullet})=0$ for any $X^{\bullet}\in dg \mathcal{X}$ and any $Y^{\bullet}\in \widetilde{\mathcal{Y}}$. Similarly, one can define $dg\mathcal{Y}$.
\end{enumerate}
\end{definition}
We denote by $\textbf{K}_{ex}(\mathcal{C})$ and $\textbf{K}(\mathcal{C})$ the corresponding homotopy categories consisting of complexes in $ex \mathcal{C}$, and $dw \mathcal{C}$, respectively.
\begin{lemma}{\rm\cite{CLY,G03}}\label{lem4.2} For the canonical cotorsion pair {\rm$(\mathcal{P}, \text{Mod}R)$}, {\rm($dw\mathcal{P}$, $dw\mathcal{P}^{\perp}$)}, {\rm($ex\mathcal{P}$, $ex\mathcal{P}^{\perp}$)} and {\rm($dg\mathcal{P}$, $\mathcal{E}$)} are the complete cotorsion pairs in {\rm$C(\text{Mod}R)$}. Moreover, they satisfy the following conditions:
\begin{enumerate}
  \item[\rm (1)] $\widetilde{\mathcal{P}}=dw\mathcal{P}\bigcap dw\mathcal{P}^{\perp}=ex\mathcal{P}\bigcap ex\mathcal{P}^{\perp}=dg\mathcal{P}\bigcap \mathcal{E}$
  \item[\rm (2)] $\mathcal{E}\bigcap dw\mathcal{P}=ex\mathcal{P}$, $dg\mathcal{P}\subseteq dw\mathcal{P}$.
\end{enumerate}

\end{lemma}
\begin{lemma}\label{lem4.3} The classes $ ex\mathcal{P}^{\perp}$ and $\mathcal{E}$ are closed under cosyzygy $\Sigma$.
\end{lemma}
\begin{proof}
Let $M^{\bullet}\in \mathcal{E}$ and $N^{\bullet}\in ex\mathcal{P}^{\perp}$. Since $C(\text{Mod}R)$ admits enough injective objects. Then, there exist two short exact sequence
\begin{equation*}
  0\longrightarrow M^{\bullet}\longrightarrow E^{\bullet}(M^{\bullet})\rightarrow \Sigma M^{\bullet}\longrightarrow 0
\end{equation*}
  \begin{equation}\label{eq4.1}
  0\longrightarrow N^{\bullet}\longrightarrow E^{\bullet}(N^{\bullet})\rightarrow \Sigma N^{\bullet}\longrightarrow 0
  \end{equation}
with $E^{\bullet}(M^{\bullet})$ and $E^{\bullet}(N^{\bullet})$ are injective complexes. It is easy to check that $\Sigma M^{\bullet}$ by the long exact sequence theorem and injective complexes are exact. Hence, $\mathcal{E}$ is closed under cosyzygy.\par
Now, it remains to show that $\Sigma N^{\bullet}\in ex\mathcal{P}^{\perp}$. Let $Q^{\bullet}$ be a complex of $ ex\mathcal{P}$. Applying the functor $\text{Hom}(Q^{\bullet},-)$ to the exact sequence (\ref{eq4.1}), we obtain a long exact sequence
$$\text{Ext}^{1}(Q^{\bullet}, N^{\bullet})\longrightarrow \text{Ext}^{1}(Q^{\bullet}, E^{\bullet}(N^{\bullet}))\longrightarrow
\text{Ext}^{1}(Q^{\bullet},  \Sigma N^{\bullet})\longrightarrow
\text{Ext}^{2}(Q^{\bullet},  N^{\bullet}).$$
Note that there is a short exact sequence
\begin{equation}\label{eq4.2}
  0\longrightarrow \Omega Q^{\bullet}\longrightarrow P^{\bullet}\longrightarrow N^{\bullet}\longrightarrow0.
\end{equation}

 Thus, we know that
$\text{Ext}^{2}(Q^{\bullet},  N^{\bullet})\cong \text{Ext}^{1}(\Omega Q^{\bullet},  N^{\bullet})$. We claim that $\Omega Q^{\bullet}\in ex\mathcal{P}$. Indeed, it is easy to see that  $\Omega Q^{\bullet}$ is exact. Moreover, the sequence (\ref{eq4.2}) is degree-wise split. Hence, each $Q_{n}$ is a projective module and so, the desired result comes. In this case, $\text{Ext}^{2}(Q^{\bullet},  N^{\bullet})=\text{Ext}^{1}(Q^{\bullet}, E^{\bullet}(N^{\bullet}))=0$. Therefore, $\text{Ext}^{1}(Q^{\bullet},  \Sigma N^{\bullet})=0$. This completes the proof.
\end{proof}
Now, we can apply one of our main results to reprove the existence of the following recollement.
\begin{proposition}{\rm\cite{CLY}}\label{prop4.4}
There exists the following recollement of triangulated category
{\rm$$\xymatrixcolsep{5pc}\xymatrix{
&\textbf{K}_{ex}(\mathcal{P})\ar[r]^{I}
&\textbf{K}(\mathcal{P})\ar@/_1.5pc/[l]_{I_{\lambda}} \ar@/^1.5pc/[l]^{I_{\del}} \ar[r]^{J}
&\textbf{D}(R)\ar@/^1.5pc/[l]^{J_{\del}}\ar@/_1.5pc/[l]_{J_{\lambda}}}.$$}
\end{proposition}
\begin{proof}
By Theorem \ref{main2}, Lemma \ref{lem4.2} and Lemma \ref{lem4.3}, we have the following recollement
$$\xymatrixcolsep{5pc}\xymatrix{
&ex\mathcal{P}/\widetilde{\mathcal{P}}\ar[r]^{I}
&dw\mathcal{P}/\widetilde{\mathcal{P}}\ar@/_1.5pc/[l]_{I_{\lambda}} \ar@/^1.5pc/[l]^{I_{\del}} \ar[r]^{J}
&dg\mathcal{P}/\widetilde{\mathcal{P}}\ar@/^1.5pc/[l]^{J_{\del}}\ar@/_1.5pc/[l]_{J_{\lambda}}}.$$
 It is well-known that there exist triangulated equivalences
$ex\mathcal{P}/\widetilde{\mathcal{P}}\cong\textbf{K}_{ex}(\mathcal{P})$, $dw\mathcal{P}/\widetilde{\mathcal{P}}\cong\textbf{K}(\mathcal{P})$ and $dg\mathcal{P}/\widetilde{\mathcal{P}}\cong\textbf{D}(R)$. Moreover, it is easy to see that the functor $I$ and $J_{\lambda}$ are triangulated functors. By adjointness, we know that the remained four functors are triangulated functors.
\end{proof}
Next, we hope to get the Krause's recollement from the cotorsion pair on the category of complexes. The following observation is very important. Let $X^{\bullet}$ be a complex. Now, we construct a short exact of complexes as follows.
 $$\xymatrix@R=0.6cm{&\vdots\ar[d]&&\vdots\ar[d]&&\vdots\ar[d]&\\
0\ar[r]&X_{n+1}\ar[rr]^{\left[
              \begin{smallmatrix}
                1\\d_{n+1}
              \end{smallmatrix}
            \right]\quad}\ar[dd]^{d_{n+1}}&&X_{n+1}\oplus X_{n}\ar[dd]^{\left[
              \begin{smallmatrix}
                0&1\\
                0&0
              \end{smallmatrix}
            \right]}\ar[rr]^-{\left[
              \begin{smallmatrix}
              -d_{n+1}&1
              \end{smallmatrix}
            \right]}&&X_{n}\ar[r]\ar[dd]^{-d_{n}}&0\\
&&&\\
0\ar[r]&X_{n}\ar[rr]^{\left[
              \begin{smallmatrix}
                1\\d_n
              \end{smallmatrix}
            \right]\quad}\ar[d]&&X_{n}\oplus X_{n-1}\ar[d]\ar[rr]^{\quad\left[
              \begin{smallmatrix}
              -d_n&1
              \end{smallmatrix}
            \right]}&&X_{n-1}\ar[r]\ar[d]&0\\
&\vdots&&\vdots&&\vdots&}$$
That is, there exists an exact sequence
 $$0\longrightarrow X^{\bullet}\xrightarrow{~\left[
              \begin{smallmatrix}
                1\\d
              \end{smallmatrix}
            \right]~} \prod_{n\in\mathbb{Z}}\overline{X_{n}}\xrightarrow{~\left[
              \begin{smallmatrix}
               -d&~1
              \end{smallmatrix}
            \right]~} S(X)\longrightarrow0 .$$
From this observation, we have the following result.
\begin{lemma}\label{lem4.5}
If $X^{\bullet}\in dg\mathcal{I}$, then there exists an isomorphism of complexes {\rm$\Sigma X^{\bullet}\cong S(X)$}. In particularly, $dg\mathcal{I}$ is closed under the cosyzygy $\Sigma$.
\end{lemma}
\begin{proof}
Since $X^{\bullet}\in dg\mathcal{I}$, $X_{n}$ are injective modules for all $n\in\mathbb{Z}$. Note that $\prod_{n\in\mathbb{Z}}\overline{X_{n}}$ is an injective complex since $\overline{X_{n}}$ are injective complexes for all $n\in\mathbb{Z}$. Then we obtain the desired result.
\end{proof}
Finally, we are in proposition to reobtain the Krause's recollement, which also reprove the Proposition 4.3 in \cite{CLY}.
\begin{proposition} For the canonical cotorsion pair {\rm$(\text{Mod}R,   \mathcal{I})$}, {\rm($^{\perp}dw\mathcal{I}$, $dw\mathcal{I}$)}, {\rm($^{\perp}ex\mathcal{I}$, $ex\mathcal{I}$)} and {\rm($\mathcal{E}$, $dg\mathcal{I}$)} are the completed cotorsion pairs in {\rm$C(\text{Mod}R)$}. Moreover, they induces Krause's recollement
{\rm$$\xymatrixcolsep{5pc}\xymatrix{
&\textbf{K}_{ex}(\mathcal{I})\ar[r]^{F}
&\textbf{K}(\mathcal{I})\ar@/_1.5pc/[l]_{F_{\lambda}} \ar@/^1.5pc/[l]^{F_{\rho}} \ar[r]^{G'}
&\textbf{D}(R)\ar@/^1.5pc/[l]^{G'_{\rho}}\ar@/_1.5pc/[l]_{G'_{\lambda}}}.$$}

\end{proposition}
\begin{proof} Following  the references \cite{CLY,G03}, we know that these three cotorsion pairs are complete and they also satisfy the following conditions:
\begin{enumerate}
  \item[(1)] $\widetilde{\mathcal{I}}={^{\perp}dw\mathcal{I}}\cap dw\mathcal{I}={^{\perp}ex\mathcal{I}}\cap ex\mathcal{I}=\mathcal{E}\cap dg\mathcal{I}$
  \item[(2)] $\mathcal{E}\cap dw\mathcal{I}=ex\mathcal{I}$, $dg\mathcal{I}\subseteq dw\mathcal{I}$.
\end{enumerate}
Now, we show that $ex\mathcal{I}$ is closed under cosyzygy. Indeed, there exists a short exact sequence $0\longrightarrow X^{\bullet}\longrightarrow E^{\bullet}(X^{\bullet})\longrightarrow \Sigma X^{\bullet}\longrightarrow 0$. Clearly, $\Sigma X^{\bullet}$ is an exact complex. It is easy to see that this sequence is degree-wise split. Thus, each component $(\Sigma X^{\bullet})_{n}$ of $\Sigma X^{\bullet}$ is injective and so, $\Sigma X^{\bullet}\in ex\mathcal{I}$.\par
By Lemma \ref{lem4.5}, $dg\mathcal{I}$ is also closed under the cosyzygy $\Sigma$. Therefore, by Theorem \ref{main1}, we know that there exists a recollement of additive categories
$$\xymatrixcolsep{5pc}\xymatrix{
&ex\mathcal{I}/\widetilde{\mathcal{I}}\ar[r]^{F}
&dw\mathcal{I}/\widetilde{\mathcal{I}}\ar@/_1.5pc/[l]_{F_{\lambda}} \ar@/^1.5pc/[l]^{F_{\rho}} \ar[r]^{G'}
&dg\mathcal{I}/\widetilde{\mathcal{I}}\ar@/^1.5pc/[l]^{G'_{\rho}}\ar@/_1.5pc/[l]_{G'_{\lambda}}}.$$
It is well-known that   there exist triangulated equivalences $ex\mathcal{I}/\widetilde{\mathcal{I}}\cong\textbf{K}_{ex}(\mathcal{I})$, $dw\mathcal{I}/\widetilde{\mathcal{I}}\cong\textbf{K}(\mathcal{I})$ and $dg\mathcal{I}/\widetilde{\mathcal{I}}\cong\textbf{D}(R)$. The remaining arguments are similar to Proposition \ref{prop4.4}.
\end{proof}

Yonggang Hu$^{a,c}$ and Panyue Zhou$^{b,c}$\\
$^a$ College of Applied Sciences, Beijing University of Technology, 100124 Beijing, P. R. China.\\
$^b$College of Mathematics, Hunan Institute of Science and Technology, 414006, Yueyang, Hunan, P. R. China.\\
$^c$D\'{e}partement de Math\'{e}matiques, Universit\'{e} de Sherbrooke, Sherbrooke,
Qu\'{e}bec J1K 2R1, Canada.\\
Yonggang Hu E-mail: \textsf{huyonggang@emails.bjut.edu.cn}\\
Panyue Zhou E-mail: \textsf{panyuezhou@163.com}

\end{document}